\documentclass[12pt]{amsart}

\usepackage[mathscr]{eucal}
\usepackage{amsmath,amssymb,amscd,amsthm}%,latexsym
\usepackage{color}
\usepackage{epsfig}

\newtheorem{proposition}{Proposition}
\theoremstyle{definition}

\theoremstyle{plane}

\def \beq{ \begin{equation} }
\def \eeq{\end{equation}}

\def \mf {\mathfrak}

\title{The non-existence of centre-of-mass and linear-momentum integrals in the curved $N$-body problem}
\begin{document}
\maketitle
\markboth{Florin Diacu}{On the integrals of the curved $N$-body problem}
\author{\begin{center}
{\bf Florin Diacu}\\
\smallskip
{\footnotesize Pacific Institute for the Mathematical Sciences\\
and\\
Department of Mathematics and Statistics\\
University of Victoria\\
P.O.~Box 3060 STN CSC\\
Victoria, BC, Canada, V8W 3R4\\
diacu@math.uvic.ca\\
}\end{center}

}

\vskip0.5cm

\begin{center}
\today
\end{center}

\begin{abstract}
We provide a class of orbits in the curved $N$-body problem for which no 
point that could play the role of the centre of mass is fixed or moves uniformly
along a geodesic. This proves that the equations of motion lack 
centre-of-mass and linear-momentum integrals.
\end{abstract}
 
%%%%%%%%
\section{Introduction}
%%%%%%%%

The importance of the centre-of-mass and linear-momentum integrals in the classical
$N$-body problem and its generalization to quasihomogeneous potentials, \cite{Diacu0}, is never enough emphasized. From integrable systems, to the Smale and Saari conjectures, \cite{Saari}, \cite{Smale}, to the study of singularities, \cite{Diacu1}, \cite{Diacu5}, these integrals have played an essential role towards answering the most basic questions of particle dynamics, in general, and celestial mechanics, in particular.
But they seem to characterize only the Euclidean space. Away from zero curvature, they disappear, probably because of the symmetry loss that occurs in the ambient space. We will show here that they don't exist in the curved $N$-body problem (the natural extension of the Newtonian $N$-body problem to spaces of constant curvature). 

This note follows from a discussion we had with Ernesto P\'erez-Chavela and Guadalupe Reyes Victoria of Mexico City. Given the absence of the above mentioned integrals in discretizations of Einstein's field equations, such as those of Levi-Civita, \cite{Civ}, \cite{Civita}, Einstein, Infeld, Hoffmann, \cite{Ein}, and Fock, \cite{Fock}, we had taken this property for granted in the curved $N$-body problem. But as we could not immediately support this claim with a counterexample, our colleagues asked if there might always exist centre-of-mass-like points in the 3-sphere, $\mathbb S^3$, and the hyperbolic 3-sphere, $\mathbb H^3$, i.e.\ space points that are either at rest or move uniformly along geodesics during the motion of the particles, as in the classical case. 

All the examples that came to mind  had centre-of-mass-like points (see Section 3), so the property we had taken for granted wasn't obvious. Therefore we began to seek orbits that failed to exhibit this feature. But before presenting the outcome of our attempts to answer this question, we need to lay some background.

The curved $N$-body problem has its roots in the work of J\'anos Bolyai, \cite{Bol},
and Nikolai Lobachevsky, \cite{Lob}. These co-founders of non-Euclidean geometry proposed, independently of each other, a 2-body problem in hyperbolic space for which the attraction is proportional to the area of the sphere of radius equal to the hyperbolic distance between the bodies. In 1870, Ernst Schering showed that a potential involving the hyperbolic cotangent of the distance agrees with this law, \cite{Sche}. Wilhelm Killing extended this idea to the sphere $\mathbb S^3$ by taking the circular tangent of the spherical distance, \cite{Kil1}, \cite{Kil10}. The topic began to attract more attention after Heinrich Liebmann proved two properties that agreed with the classical Kepler problem: the potential is a harmonic function in 3-space and every bounded orbit is closed, \cite{Ber}, \cite{Lie1}, \cite{Lie2}, \cite{Lie3}. Then Ernest Schr\"odinger and his followers exploited these ideas at the quantum level, \cite{Schr}, \cite{Inf}, \cite{InfS}. Starting with the 1990s, the Russian school of celestial mechanics further developed the 2-body case, \cite{Koz}, \cite{Koz2}, \cite{Shc}, \cite{Shc1}, \cite{Shc2}. 

In some recent work, we generalized the Schering-Killing potential to any number $N\ge 2$ of bodies, \cite{Diacu1}, \cite{Diacu3}, \cite{Diacu-mono}, \cite{Diacu-memoirs},  \cite{Diacu2}, \cite{Diacu4}, \cite{Diacu5}, \cite{Diacu6}, and used the results obtained there to argue that, for distances comparable to those of our solar system, the geometry of the physical space is Euclidean. Our method exploited the connection between geometry and dynamics, thus going beyond Gauss's attempt to determine the nature of the physical space through topographic measurements. For more details on this subject and an extensive bibliography, see \cite{Diacu-mono}, \cite{Diacu-memoirs}.

In this note we prove that, for a certain class of solutions of the curved $N$-body
problem in $\mathbb S^3$, there are no centre-of-mass-like points. In other words, 
we show that forces cannot keep any such point at rest or make it move uniformly along a great circle of $\mathbb S^3$, results that occur in Sections \ref{counterexamples} and \ref{class}. For positive curvature, our choice of $\mathbb S^3$ is not restrictive since a suitable coordinate and time-rescaling transformation makes the curvature parameter vanish from the equations of motion, \cite{Diacu-mono}, \cite{Diacu7}.  In Section \ref{motion} we introduce these differential equations and the known integrals of motion, in Section \ref{examples} we give examples of orbits with centre-of-mass-like points, and in Section \ref{compl} we set the geometric background, necessary for understanding the orbits we describe in Sections \ref{counterexamples} and \ref{class}. 

Unfortunately, a similar example in $\mathbb H^3$ still eludes us, and we don't know the size of the set of orbits that have centre-of-mass-like points in $\mathbb S^3$ or $\mathbb H^3$, but expect it to be negligible from both the measure theoretical and the topological point of view. So apart from the result presented in this note, which proves the non-existence of these integrals for positive curvature, the issue my colleagues raised leads to several interesting questions, all of which still await answers.

%%%%%%%%
\section{Equations of motion}\label{motion}
%%%%%%%%

Consider $N$ point particles (bodies) of masses $m_1,\dots, m_N>0$ moving in 
$\mathbb{S}^3$ (thought as embedded in the Euclidean space $\mathbb R^4$) or $\mathbb H^3$ (thought as embedded in the Minkowski space $\mathbb R^{3,1}$),
where
$$
{\mathbb S}^3=\{(w,x,y,z)\ | \ w^2+x^2+y^2+z^2=1\},
$$
$$
{\mathbb H}^3=\{(w,x,y,z)\ | \ w^2+x^2+y^2-z^2=-1, \ z>0\}.
$$
Then the configuration of the system is described by the $4N$-dimensional vector
$$
{\bf q}=({\bf q}_1,\dots,{\bf q}_N),
$$
where ${\bf q}_i=(w_i,x_i,y_i,z_i), i=\overline{1,N}$, denote the position vectors of the bodies.  The equations of motion are given by the second-order system
\begin{equation}
\label{both}
\ddot{\bf q}_i=\sum_{j=1,j\ne i}^N\frac{m_j[{\bf q}_j-\sigma({\bf q}_i\odot {\bf q}_j){\bf q}_i]}{[\sigma-\sigma({\bf q}_i\odot {\bf q}_j)^2]^{3/2}}-\sigma(\dot{\bf q}_i\odot \dot{\bf q}_i){\bf q}_i, \ \ i=\overline{1,N},
\end{equation}
with constraints
\begin{equation}
{\bf q}_i\odot{\bf q_i}=\sigma, \ \ {\bf q}_i\odot\dot{\bf q}_i=0, \ \ i=\overline{1,N},
\end{equation}
where $\odot$ is the standard inner product $\cdot$ of signature $(+,+,+,+)$ in $\mathbb S^3$ or the Lorentz inner product $\boxdot$ of signature $(+,+,+,-)$ in $\mathbb H^3$, and
$$
\sigma=
\begin{cases}
+1,\ \ {\rm in}\ \ {\mathbb S^3},\cr
-1,\ \ {\rm in}\ \ {\mathbb H^3},
\end{cases}
$$
denotes the signum function, \cite{Diacu-mono}, \cite{Diacu-memoirs}. System \eqref{both} has dimension $6N$. The force acting on each body has a gravitational component (the above sum) and a term (involving the velocities) that corresponds to the constraints.

As a consequence of Noether's theorem, system \eqref{both} has the scalar integral of energy,
$$
T({\bf q},\dot{\bf q})-U({\bf q})=h,
$$
where 
$$U({\bf q})=\sum_{1\le i<j\le N}\frac{\sigma m_im_j{\bf q}_i\odot{\bf q}_j}{[\sigma-\sigma({\bf q}_i\odot{q}_j)^2]^{3/2}}$$ 
is the force function ($-U$ representing the potential),
$$
T({\bf q},\dot{\bf q})=\frac{1}{2}\sum_{i=1}^Nm_i(\dot{\bf q}_i\odot\dot{\bf q}_i)
(\sigma{\bf q}_i\odot{\bf q}_i)
$$
is the kinetic energy, with $h$ representing an integration constant; and the 6-dimensional integral of the total angular momentum,
$$
\sum_{i=1}^Nm_i{\bf q}_i\wedge\dot{\bf q}_i={\bf c},
$$
where $\wedge$ is the wedge product and ${\bf c}=(c_{wx},c_{wy},c_{wz},c_{xy},c_{xz},c_{yz})$ denotes an integration vector, each component measuring the rotation of the system about the origin of the frame relative to the plane corresponding to the bottom indices.

But it is far from clear whether the centre-of-mass and linear-momentum integrals exist for system \eqref{both}. In Euclidean space, these
integrals are obtained by showing that $\sum_{i=1}^Nm_i\ddot{\bf q}_i={\bf 0}$. This identity is not satisfied in the curved $N$-body problem because $\frac{1}{M}\sum_{i=1}^Nm_i{\bf q}_i$,
where $M=m_1+\dots+m_N$, does not even belong to $\mathbb S^3$ or $\mathbb H^3$. However, it may be possible that there are points in $\mathbb S^3$ or $\mathbb H^3$ that behave like a centre of mass, i.e.\ the forces acting on them cancel each other within the manifold or make those points move uniformly along a geodesic. The next section will provide a couple of examples in which such centre-of-mass-like points exist.

%%%%%%%%
\section{Orbits with centre-of-mass-like points} \label{examples} 
%%%%%%%%

The orbits with centre-of-mass-like points we will further present are the Lagrangian
elliptic relative equilibria of equal masses in $\mathbb S^3$ and the Eulerian
hyperbolic relative equilibria of equal masses in $\mathbb H^3$. More examples
having this property can be found in \cite{Diacu-mono}, \cite{Diacu-memoirs}.

%%%%
\subsection{Lagrangian elliptic relative equilibria in $\mathbb S^3$}
%%%%

A straightforward computation shows that for $N=3$, $\sigma=1$, and $m_1=m_2=m_3=:m$, system \eqref{both} has orbits of the form
\begin{equation}
{\bf q}=({\bf q}_1,{\bf q}_2,{\bf q}_3), \ {\bf q}_i=(w_i,x_i,y_i,z_i),\ i=1,2,3,
\label{0elliptic1}
\end{equation}
\begin{align*}
w_1(t)&=r\cos\omega t,& x_1(t)&=r\sin\omega t,\\
 y_1(t)&=y\ ({\rm constant}),& z_1(t)&=z\ ({\rm constant}),\\
w_2(t)&=r\cos(\omega t+2\pi/3),& x_2(t)&=r\sin(\omega t+2\pi/3),\\
y_2(t)&=y\ ({\rm constant}),& z_2(t)&=z\ ({\rm constant}),\\
w_3(t)&=r\cos(\omega t+4\pi/3),& x_3(t)&=r\sin(\omega t+4\pi/3),\\
y_3(t)&=y\ ({\rm constant}),& z_3(t)&=z\ ({\rm constant}),
\end{align*}
with $r^2+y^2+z^2=1, r\in(0,1),$ and
$$
\omega^2=\frac{8m}{\sqrt{3}r^3(4-3 r^2)^{3/2}}.
$$
These are Lagrangian elliptic relative equilibria, i.e.\ orbits for which the bodies lie at the vertices of a rotating equilateral triangle, such that the mutual distances between bodies remain constant during the motion. The forces acting on points of the form ${\bf q}_0=(0,0,y_0,z_0)\in\mathbb S^3$, i.e.\ with $y_0^2+z_0^2=1$, cancel each other within $\mathbb S^3$. Moreover, the spherical distance, $d_{\mathbb S^3}$, from ${\bf q}_0$ to $m_1, m_2$, and $m_3$ is the same:
$$
d_{\mathbb S^3}({\bf q}_0,{\bf q}_i)=\cos^{-1}({\bf q}_0\cdot{\bf q}_i)=\cos^{-1}(y_0y+z_0z),   \ i=1,2,3.
$$

Therefore all these points (infinitely many) are centre-of-mass-like. The condition that the masses are equal is sharp: such equilateral orbits don't exist if the masses are distinct, but it can be shown that orbits of nonequal masses exist for scalene triangles, \cite{Diacu3}.

If we restrict the motion to the great 2-sphere 
$$
{\bf S}_w^2=\{(0,x,y,z)\ | \ x^2+y^2+z^2=1\},
$$
then there are only two centre-of-mass-like points, the north pole $(0,1,0,0)$
and the south pole $(0,-1,0,0)$ of ${\bf S}_w^2$. Nevertheless, the situation is still different from the Euclidean case, where the centre of mass is unique.

%%%%
\subsection{Eulerian hyperbolic relative equilibria in $\mathbb H^3$}
%%%%

A straightforward computation shows that for $N=3$, $\sigma=-1$, and $m_1=m_2=m_3=:m$, system \eqref{both} has orbits of the form
\begin{equation}
{\bf q}=({\bf q}_1, {\bf q}_2, {\bf q}_3), \ {\bf q}_i=(w_i,x_i,y_i,z_i), \ i=1,2,3,
\label{hyp-4}
\end{equation}
\begin{align*}
w_1&=0,& x_1&=0,& y_1&=\sinh\beta t,& z_1&=\cosh\beta t,\\
w_2&=0,& x_2&=x\ {\rm (constant)},& y_2&=\eta\sinh\beta t,& z_2&=\eta\cosh\beta t,\\
w_3&=0,& x_3&=-x\ {\rm (constant)},& y_3&=\eta\sinh\beta t,& z_3&=\eta\cosh\beta t,
\end{align*}
with $x^2-\eta^2=-1,\ \eta>1$, and 
$$
\beta^2=\frac{1+4\eta^2}{4\eta^3(\eta^2-1)^{3/2}}.
$$
These are Eulerian hyperbolic relative equilibria, orbits lying on a geodesic of $\mathbb H^3$ that rotates hyperbolically, which means that the mutual distances between bodies remain constant during the motion.

Points of the form ${\bf q}_*=(w_*,0,\rho_*\sinh\beta t,\rho_*\cosh\beta t)$, with $w_*^2-\rho_*^2=-1$, move uniformly along the geodesic $w=w_*, x=0$ in $\mathbb H^3$. The hyperbolic distance, $d_{\mathbb H^3}$, between ${\bf q}_*$ and $m_1$ is constant along the motion because
$$
d_{\mathbb H^3}({\bf q}_*,{\bf q}_1)=\cosh^{-1}(-{\bf q}_*\boxdot{\bf q}_1)=\cosh^{-1}\rho_*.
$$
The same holds for the distance between ${\bf q}_*$ and $m_2$ or $m_3$,
$$
d_{\mathbb H^3}({\bf q}_*,{\bf q}_i)=\cosh^{-1}(-{\bf q}_*\boxdot{\bf q}_i)=\cosh^{-1}\eta\rho_*, \  i=2,3.
$$
Consequently ${\bf q}_*$ is centre-of-mass-like. For $w_*=0$, which yields $\rho_*=1$, the point ${\bf q}_*$ overlaps with $m_1$. This fact shows that when we restrict the motion to the great hyperbolic 2-sphere,
$$
{\bf H}_w^2=\{(0,x,y,z)\ | \ x^2+y^2-z^2=-1\},
$$
there is only one centre-of-mass-like point, which is identical to $m_1$.

%%%%%%%%
\section{Complementary circles in $\mathbb S^3$}\label{compl}
%%%%%%%%

In order to produce a class of orbits with no centre-of-mass-like points, we will introduce the following geometric concept. Two great circles $C_1$ and $C_2$ of two distinct great 2-spheres of $\mathbb S^3$ are called complementary if there is a coordinate system in which they can be represented as 
$$
C_1=\{(w,x,y,z)\ | \ w=x=0\ {\rm and} \ y^2+z^2=1\},
$$
$$
C_2=\{(w,x,y,z)\ | \ y=z=0\ {\rm and} \ w^2+x^2=1\}.
$$
In topological terms, $C_1$ and $C_2$ form a Hopf link in a Hopf fibration, which is a map
$$
{\mathcal H}\colon{\mathbb S}^3\to{\mathbb S}^2, \ h(w,x,y,z)=(w^2+x^2-y^2-z^2,2(wz+xy),2(xz-wy))
$$
that takes circles of $\mathbb S^3$ to points of $\mathbb S^2$,
\cite{Hopf}, \cite{Lyons}. In particular, ${\mathcal H}(C_1)=(1,0,0)$ and ${\mathcal H}(C_2)=(-1,0,0)$. Using the stereographic projection, it can be shown that the circles $C_1$ and $C_2$ are linked (like the adjacent rings of a chain), hence the name of the pair, \cite{Lyons}.

The spherical distance  between two complementary circles is constant. Indeed, if ${\bf a}=(0,0,y,z)\in C_1$ and ${\bf b}=(w,x,0,0)\in C_2$, then ${\bf a}\cdot{\bf b}=0$, so
$d_{\mathbb S^3}({\bf a},{\bf b})=\pi/2.$
This unexpected fact turns out to be even more surprising from the dynamical point of view: the magnitude of the gravitational force (but not its direction) between a body lying on a great circle and a body lying on the complementary great circle is the same, no matter where the bodies are on their respective circles. 

%%%%%%%%
\section{Non-existence of linear-momentum and centre-of-mass integrals}\label{counterexamples}
%%%%%%%%

We can now construct an example in the 6-body problem in $\mathbb S^3$ in which three bodies of equal masses move along a great circle at the vertices of an equilateral triangle, while the other 3 bodies of masses equal to those of the previous bodies move along a complementary circle of another great sphere, also at the vertices of an equilateral triangle. So take the masses 
$m_1=m_2=m_3=m_4=m_5=m_6=:m>0$ and the frequencies $\alpha, \beta\ne 0$, which, in general, are distinct, $\alpha\ne\beta$.
Then a candidate for a solution as described above has the form
\begin{equation}
{\bf q}=({\bf q}_1,{\bf q}_2, {\bf q}_3, {\bf q}_4, {\bf q}_5, {\bf q}_6),\ {\bf q}_i=(w_i,x_i,y_i,z_i),\ i=\overline{1,6},
\label{sol}
\end{equation}
\begin{align*}
w_1&=\cos\alpha t, & x_1&=\sin\alpha t,\\ 
y_1&=0, & z_1&=0,\displaybreak[0]\\\
w_2&=\cos(\alpha t+2\pi/3), & x_2&=\sin(\alpha t +2\pi/3),\\ y_2&=0, & z_2&=0,\displaybreak[0]\\
w_3&=\cos(\alpha t+4\pi/3), & x_3&=\sin(\alpha t +4\pi/3),\\ y_3&=0, & z_3&=0,\displaybreak[0]\\
w_4&=0, & x_4&=0, \\ y_4&=\cos\beta t, & z_4&=\sin\beta t,\displaybreak[0]\\
w_5&=0, & x_5&=0,\\ y_5&=\cos(\beta t+2\pi/3), & z_5&=\sin(\beta t+2\pi/3),\displaybreak[0]\\
w_6&=0, & x_6&=0,\\ y_6&=\cos(\beta t+4\pi/3), & z_6&=\sin(\beta t+4\pi/3).
\end{align*}
For $t=0$, we obtain a fixed-point configuration, ${\bf q}(0)$, specific to ${\mathbb S}^3$, in the sense that there is no 2-sphere that contains it:
\begin{align*}
w_1&=0,& x_1&=1,& y_1&=0,& z_1&=0,\displaybreak[0]\\\
w_2&=-\frac{1}{2},& x_2&=\frac{\sqrt{3}}{2},& y_2&=0,& z_2&=0,\displaybreak[0]\\
w_3&=-\frac{1}{2},& x_3&=-\frac{\sqrt{3}}{2},& y_3&=0,& z_3&=0,\displaybreak[0]\\
w_4&=0,& x_4&=0,& y_4&=1,& z_4&=0,\displaybreak[0]\\
w_5&=0,& x_5&=0,& y_5&=-\frac{1}{2},& z_5&=\frac{\sqrt{3}}{2},\displaybreak[0]\\
w_6&=0,& x_6&=0,& y_6&=-\frac{1}{2},& z_6&=-\frac{\sqrt{3}}{2}.
\end{align*}
A long but straightforward computation shows that \eqref{sol} is solution
of system \eqref{both} for $N=6$, $\sigma=1$, and any $\alpha,\beta\ne 0$. If $\alpha$ and $\beta$ are rational multiples of $2\pi$, a set of frequency pairs that has measure zero in $\mathbb R^2$, the corresponding orbits are periodic. But in general $\alpha/\beta$ is a irrational, so the orbits are quasiperiodic.
The angular momentum constants are 
$$
c_{wx}=3m\alpha\ne 0,\ \ c_{yz}=3m\beta\ne 0, \ \
c_{wy}=c_{wz}=c_{xy}=c_{xz}=0,
$$
which means that the rotation of the system takes place around the origin of the coordinate system only relative to the planes $wx$ and $yz$.

We can now state and prove the main result of this note.

%%%%% PROPOSITION
\begin{proposition}
For a solution \eqref{sol} of system \eqref{both}, with $N=6,\ \sigma=1$, and $\alpha/\beta$ irrational, there is no point in $\mathbb S^3$ such that the forces acting on it vanish in $\mathbb S^3$\! or make it move uniformly along a geodesic.
\end{proposition} 
%%%%%
\begin{proof}
A solution \eqref{sol} of system \eqref{both} can be viewed as obtained from the action of an element of the Lie group $SO(4)$ on the fixed-point configuration ${\bf q}(0)$. Without loss of generality, we can choose a suitable basis to give this element of the group the matrix representation 
\begin{equation}\label{matrix}
A=\begin{pmatrix}
\cos\gamma t & -\sin\gamma t & 0 & 0\\
\sin\gamma t & \cos\gamma t & 0 & 0\\
0 & 0 & \cos\delta t & -\sin\delta t\\
0 & 0 & \sin\delta t & \cos\delta t
\end{pmatrix},
\end{equation}
with $\gamma,\delta\ne 0$ compatible with $\alpha,\beta\ne 0$. Then the matrix $A$ rotates all the points of $\mathbb S^3$, including the bodies $m_1,\dots, m_6$, such that the latter form a solution of the equations of motion.

The only point of $\mathbb R^4$ that stays put under the action of  $A$, when $\gamma,\delta\ne 0$, is the origin, $\bf 0$, of the coordinate system. But ${\bf 0}\notin\mathbb S^3$, so no point in $\mathbb S^3$ stays fixed.

We still need to prove that no candidate for a centre-of-mass-like point of this 6-body problem  can move uniformly along a geodesic of $\mathbb S^3$ under the action of $A$. For this purpose, consider a point ${\mathfrak q}_0=(\mf w_0, \mf x_0, \mf y_0, \mf z_0)$
in $\mathbb S^3$.
Then the components of the vector $A{\mathfrak q}_0^T$, where the upper $^T$\! represents the transposed, are:
\begin{align*}
\mf w(t)&=\mf w_0\cos\gamma t -\mf x_0\sin\gamma t, &  \mf x(t)&=\mf w_0\sin\gamma t + \mf x_0\cos\gamma t,\\
\mf y(t)&=\mf y_0\cos\delta t -\mf z_0\sin\delta t, &  \mf z(t)&=\mf y_0\sin\delta t + \mf z_0\cos\delta t.
\end{align*}
But since $\gamma$ and $\delta$ are compatible with $\alpha$ and $\beta$,
we have
$$
\mf w(t)=r_0\cos\alpha t,\ \mf x(t)=r_0\sin\alpha t,\ \mf y(t)=\rho_0\cos\beta t, \ \mf z(t)=\rho_0\sin\beta t,
$$
where $r_0=(\mf w_0^2+ \mf x_0^2)^{1/2}$ and $\rho_0=(\mf y_0^2+ \mf z_0^2)^{1/2}$, with
$r_0^2+\rho_0^2=1$.

To determine when the curve ${\mathfrak q}=(\mf{w,x,y,z})$, described by the point $
(\mf w_0, \mf x_0, \mf y_0, \mf z_0)$ under the action of $A$, describes a geodesic, it is necessary that the point moves on a great sphere of $\mathbb S^3$, i.e.\ the function $\mf q$ that draws the trajectory satisfies the conditions
$$
\mf w^2+ \mf x^2+ \mf y^2+ \mf z^2=1\ \ {\rm and}\ \ A\mf w+B \mf x+C \mf y+D\mf z=0,
$$
where not all the real coefficients $A,B,C,D$ are zero. The first condition, which asks that $\mf q$ lies
in $\mathbb S^3$, is obviously satisfied. The second condition, which requires that $\mf q$ lies in a hyperplane passing through the origin of the coordinate system, translates into
$$
Ar_0\cos\alpha t + Br_0\sin\alpha t + C\rho_0\cos\beta t + D\rho_0\sin\beta t =0.
$$
Since $\alpha/\beta$ is not rational, the above equality holds for all $t$ only if
$$
r_0(A\cos\alpha t + B\sin\alpha t)=\rho_0(C\cos\beta t + D\sin\beta t)=0,
$$
which is equivalent to either 

(i) $r_0=0$ and $C=D=0$ or 

(ii) $\rho_0=0$ and $A=B=0$.

\noindent As $r_0=0$ implies $w_0=x_0=0$ and $\rho_0=1$, and $\rho_0=0$ implies $y_0=z_0=0$ and $r_0=1$, the trajectory $\mathfrak q$ of ${\mathfrak q}_0$ under the action of $A$ can be either in agreement with (i) or with (ii), but not with both, i.e.\ ${\mf q}(t)$ has the form
$$
\begin{pmatrix}
0\\
0\\
\cos\beta t\\
\sin\beta t
\end{pmatrix}
\ \ {\rm or}\ \
\begin{pmatrix}
\cos\alpha t\\
\sin\alpha t\\
0\\
0
\end{pmatrix}.
$$
Therefore the only geodesics lying on great spheres that remain invariant (globally, not point-wise) under the action of $A$, with $\gamma,\delta\ne 0$, are the complementary great circles 
$$
C_1=\{(0,0,y,z)\ | \ y^2+z^2=1\}\ \ {\rm and}\ \  C_2=\{(w,x,0,0)\ | \ w^2+x^2=1\}.
$$

Any point moving uniformly on $C_1$ is equidistant from any point of $C_2$, so it could act like a centre of mass for the bodies rotating on $C_2$, but it does not qualify as
a centre of mass for the bodies rotating on $C_1$, so it cannot be a centre of mass for the entire system. The same is true if the roles of $C_1$ and $C_2$ are interchanged. Since no other points rotate on geodesics, these remarks complete the proof.
\end{proof}

%%%%%%%%
\section{A class of counterexamples}\label{class}
%%%%%%%%

The example presented in Section 5 can be extended to a larger class of orbits with no centre-of-mass-like points. Consider first the $(N+M)$-body problem (with $N,M\ge 3$ and odd) of equal masses, in which $N$ bodies rotate along a great circle of a great sphere at the vertices of a regular $N$-gon, while the other $M$ bodies rotate along a complementary great circle of another great sphere at the vertices of a regular $M$-gon. The same as in the 6-body problem discussed above, the rotation takes place around the origin of the coordinate system only relative to two out of six reference planes. The condition that $N$ and $M$ are odd is imposed to avoid antipodal configurations, which introduce singularities, \cite{Diacu1}, \cite{Diacu5}. The existence of the orbits we just described follows from \cite{Diacu3}, \cite{Diacu-mono}, \cite{Diacu-memoirs}. The mass-equality condition can be relaxed, a case in which the polygons are not regular anymore. For suitable polygonal shapes and well chosen masses, we can also eliminate the condition that $M$ and $N$ are odd. The non-existence proof given in Proposition 1 works similarly in the general case. Consequently the above class of orbits covers any finite number of bodies, and shows that, in the curved $N$-body problem in $\mathbb S^3$ there are no centre-of-mass and linear-momentum integrals.

\medskip

\noindent{\bf Acknowledgment.} The research presented in this note was supported in part by a Discovery Grant the author received from NSERC of Canada.

%%%%%%
%%%%%%          BIBLIOGRAPHY
%%%%%%

\end{document}